\documentclass[11pt,reqno,a4paper]{amsart}

\usepackage{amsmath,amsthm,amssymb,mathtools,bbm}
\usepackage{tikz}
\usepackage{bookmark}
\usepackage{hyperref}
\hypersetup{pdfstartview={FitH}}

\theoremstyle{plain}
\newtheorem{theorem}{Theorem}
\newtheorem{proposition}[theorem]{Proposition}
\newtheorem{corollary}[theorem]{Corollary}
\newtheorem{lemma}[theorem]{Lemma}

\newtheorem*{question}{Question}
\newtheorem*{problem}{Problem}

\theoremstyle{definition}

\numberwithin{equation}{section}

\renewcommand{\leq}{\leqslant}
\renewcommand{\le}{\leqslant}
\renewcommand{\geq}{\geqslant}
\renewcommand{\ge}{\geqslant}

\newcommand{\T}{\mathbb{T}}
\newcommand{\Z}{\mathbb{Z}}
\newcommand{\R}{\mathbb{R}}

\newcommand{\e}{\mathrm{e}}
\newcommand{\dist}{\operatorname{dist}}
\newcommand{\card}{\operatorname{card}}
\newcommand{\norm}[1]{\lVert #1 \rVert}

\begin{document}

\title[Near-optimal density theorems for large point configurations]{Near-optimal density theorems for large dilates of large point configurations}

\author[Vjekoslav Kova\v{c}]{Vjekoslav Kova\v{c}}
\address{Department of Mathematics, Faculty of Science, University of Zagreb, Bijeni\v{c}ka cesta 30, 10000 Zagreb, Croatia}
\email{vjekovac@math.hr}

\author[Adian Anibal Santos Sep\v{c}i\'{c}]{Adian Anibal Santos Sep\v{c}i\'{c}}
\address{Bachelor's Program in Mathematics, Faculty of Science, University of Zagreb, Bijeni\v{c}ka cesta 30, 10000 Zagreb, Croatia}
\email{adian.anibal.santos.sepcic@student.math.hr}



\begin{abstract}
We study density thresholds that force a measurable set $E\subseteq \R^d$ to contain all sufficiently large similar copies of every $n$-point configuration. We prove a lower bound of the form $1-O((\log n)/n)$, which matches the known upper bound up to the logarithmic factor, thus essentially resolving a problem posed by Falconer, Yavicoli, and the first author of the present paper. We also study the same problem for embeddings of $n$-point configurations into $\R^d$ equipped with the $\ell^p$ norm, obtaining an asymptotically sharp bound $1-1/n+o(1/n)$, as soon as $p\in(1,\infty)\setminus\{2\}$. In the proof of the former estimate we use equidistribution of polynomial sequences modulo $1$ combined with probabilistic thinning. The proof of the latter estimate relies on the geometry of the $\ell^p$ spaces for $p\neq2$.
\end{abstract}

\maketitle



\section{Introduction}

A classical theme in Euclidean Ramsey theory asks which large-scale finite point configurations must occur in measurable sets of positive (upper) density. Bourgain's theorem on large simplices \cite{Bou86} initiated a long line of work on similar density theorems. Notable examples are the results on product configurations \cite{LM16:prod,LM19:hypergraphs,DK22,DS25}, distance graphs \cite{LM20,KP24}, and arithmetic progressions and their generalizations \cite{CMP17,DKR18,DK21,DK22}. A very brief survey of the topic can be found in \cite{Kovac:survey}, while the reader can consult the introductory sections of the recent papers \cite{Kov22} and \cite{Kov26} for a more detailed overview of the literature.

For a measurable set $E\subseteq\R^d$ we write
\[ \overline{\textup{d}}(E):=\limsup_{R\to\infty}\frac{|E\cap[-R/2,R/2]^d|}{R^d}, \qquad
\underline{\textup{d}}(E):=\liminf_{R\to\infty}\frac{|E\cap[-R/2,R/2]^d|}{R^d} \]
for its \emph{upper} and \emph{lower asymptotic density}, respectively.
If the actual limit exists, we write it as $\textup{d}(E)$ and call it simply the \emph{density} of $E$.
The \emph{upper Banach density} of $E$ is
\[ \overline{\rho}(E):=\lim_{R\to\infty} \sup_{x\in\R^d} \frac{|E\cap(x+[0,R]^d)|}{R^d}, \]
and it is well known that the limit exists.
(Often, the notation $\overline{\delta}(E)$ is used instead.)
Clearly,
\[ \underline{\textup{d}}(E)\leq \overline{\textup{d}}(E)\leq \overline{\rho}(E). \]
Moreover, the latter notion is so flexible that the cube $[0,1]^d$, dilates of which are considered, can be replaced with a ball $\textup{B}(0,1)$, or any other compact convex set with nonempty interior; see, for example, \cite{BPPhD}.
Thus, for instance, we have
\[
\overline{\rho}(E) = \lim_{R\to\infty} \sup_{x\in\R^d} \frac{|E\cap \textup{B}(x,R)|}{|\textup{B}(x,R)|}
\]
as an exact equality, i.e., without any loss of constants.

Bourgain \cite[Section~1]{Bou86} and Graham \cite[Section~4]{Gra94} found examples of point configurations $P$ and sets $E\subseteq\R^d$ of positive density for which there exist arbitrarily large ``scales'' $r\in(0,\infty)$ such that $E$ does not contain an isometric copy of $r P$.
Motivated by this, Falconer, Yavicoli, and the first author \cite{FKY22} studied the critical density of $E$ with this property in terms of the pattern size $n=\card P$.

\begin{question}[{\cite[Question~1.6]{FKY22}}]
What is the smallest $0\leq \rho_{\min}(d,n) \leq 1$ such that every measurable set $E \subseteq \R^d$ of upper Banach density larger than $\rho_{\min}(d, n)$ contains all sufficiently large-scale similar copies of all $n$-point patterns?
\end{question}

The same paper \cite[Theorems~1.4 and~1.5]{FKY22} showed that
\begin{equation}\label{eq:FKYbound}
1 - \frac{10\log n}{n^{1/5}} \leq \rho_{\min}(d,n) \leq 1 - \frac{1}{n-1}.
\end{equation}
In fact, a trivial upper estimate, 
\[ \rho_{\min}(d,n) \leq 1 - \frac{1}{n}, \]
can be obtained by pigeonholing and merely considering the translates of $P$; it was noted in \cite[Proposition~1.3]{FKY22} that
\begin{quote}
\emph{if $P\subseteq\R^d$ has $n$ elements and we take any $r>0$, then every set $E\subseteq\R^d$ of upper Banach density larger than $1-1/n$ contains a translated copy of $r P$.}
\end{quote}
In \cite{FKY22} it took significant effort to improve this upper bound to just $1-1/(n-1)$, which suggested that the true behavior of $\rho_{\min}(d,n)$ could be of the form $1-O(1/n)$.
The following theorem gives a quantitative improvement of the lower bound in \eqref{eq:FKYbound}, which is almost sharp, up to a logarithmic factor.

\begin{theorem}
\label{theorem:l2}
There exists an absolute constant $C>0$ with the following property.
For every integer $d\ge 1$ and every sufficiently large $n$ there exist
\begin{itemize}
    \item a measurable set $E\subseteq\R^d$,
    \item an $n$-point configuration $P\subseteq\R^d$ contained in a line,
    \item and a sequence of positive numbers $(r_j)_{j}$ with $r_j\to\infty$,
\end{itemize}
such that
\[
\textup{d}(E)\ge 1-C\frac{\log n}{n}
\]
and $E$ contains no isometric copy of $r_jP$ for any $j$.
\end{theorem}

This improves the lower bound on $\rho_{\min}(d,n)$ to the asymptotic estimate
\[ \rho_{\min}(d,n) \geq 1-C\frac{\log n}{n} \]
for all sufficiently large $n$. It does not, however, reach the conjectured $1-O(1/n)$ behavior.
On the other hand, we also found an entirely elementary argument that gives a weaker bound
\[ \rho_{\min}(d,n) \geq 1-\frac{C}{n^{1/2}}, \]
which already improves \eqref{eq:FKYbound}. We sketch this construction in the very brief Section~\ref{sec:elementary}, even though both approaches start with the same annular type of set $E$ given in Section~\ref{sec:reduction}.

Instead of proving Theorem~\ref{theorem:l2} directly, we will actually prove its generalization to the space $\R^d$ equipped with the $\ell^p$ norm,
\[ \|x\|_p := \Big( \sum_{i=1}^d |x_i|^p \Big)^{1/p} \]
for $x=(x_1,x_2,\ldots,x_d)\in\R^d$. Here we are only interested in the exponents $p\in(1,\infty)$.
Density theorems involving lengths measured in $\ell^p$ norms were first studied in the work of Cook, Magyar, and Pramanik \cite{CMP17}, and later also in \cite{DK21,DK22}, motivated by the fact that the results of those papers were not available in the Euclidean norm (that is, the $\ell^2$ norm).

\begin{theorem}
\label{theorem:lp}
Fix integers $d\ge 1$ and $p\ge 2$. There exists a constant $C_{d,p}\in(0,\infty)$ with the following property. 
For every sufficiently large integer $n$ there exist
\begin{itemize}
    \item a measurable set $E\subseteq\R^d$,
    \item an $n$-point configuration $P\subseteq\R^d$ contained in a line,
    \item and a sequence of positive numbers $(r_j)_{j}$ with $r_j\to\infty$,
\end{itemize}
such that
\begin{equation}\label{eq:mainlowerest}
\underline{\textup{d}}(E)\ge 1-C_{d,p}\frac{\log n}{n}
\end{equation}
and $E$ contains no $\ell^p$-isometric copy of $r_jP$ for any $j$.

Moreover, if $p$ is even, then the constant $C_{d,p}$ in \eqref{eq:mainlowerest} can be chosen independently of the dimension $d$ and the set $E$ can be chosen to have the actual density $\textup{d}(E)$.
\end{theorem}

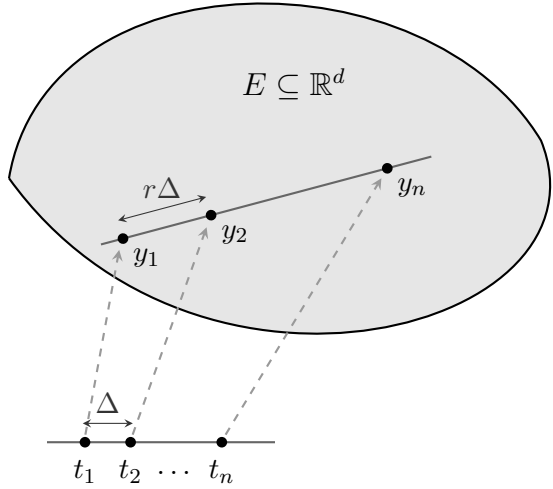
\begin{figure}
\begin{center}
\begin{tikzpicture}[scale=1.0, >=stealth]
    \draw[thick, black, fill=black!10] 
        (-0.5,2) .. controls (0,5) and (5,5) .. 
        (6.5,2.5) .. controls (7.5,0) and (2,-1.5) .. 
        (-0.5,2);
    \node[font=\large, black] at (3.25, 3.25) {$E \subseteq \mathbb{R}^d$};
    \begin{scope}[shift={(0.5,-2.5)}]
        \draw[thick, black!60] (-0.5, 1) -- (2.5, 1);
        \fill[black] (0,1) circle (2pt) node[below=3pt, black] {$t_1$};
        \fill[black] (0.6,1) circle (2pt) node[below=3pt, black] {$t_2$};
        \fill[black] (1.8,1) circle (2pt) node[below=3pt, black] {$t_n$};
        \node[black] at (1.2, 0.6) {$\dots$};
        \draw[<->, black!80, thin] (-0.03, 1.25) -- (0.63, 1.25) node[midway, above] {$\Delta$};
    \end{scope}
    \begin{scope}[shift={(1.0, 1.2)}, rotate=15]
        \draw[thick, black!60] (-0.3, 0) -- (4.2, 0);
        \fill[black] (0,0) circle (2pt) node[below right, black] {$y_1$};
        \fill[black] (1.2,0) circle (2pt) node[below right, black] {$y_2$};
        \fill[black] (3.6,0) circle (2pt) node[below right, black] {$y_n$};
        \draw[<->, black!80, thin] (0, 0.25) -- (1.2, 0.25) node[midway, above] {$r\Delta$};
    \end{scope}
    \draw[->, black!40, dashed, thick] (0.505, -1.4) -- (0.95, 1.05);
    \draw[->, black!40, dashed, thick] (1.127, -1.4) -- (2.1, 1.35);
    \draw[->, black!40, dashed, thick] (2.35, -1.4) -- (4.4, 2.0);
\end{tikzpicture}
\end{center}
\caption{$\{y_1,\ldots,y_n\}$ is an isometric copy of $\{t_1,\ldots,t_n\}$ dilated by $r$.}
\label{fig:dilcopy}
\end{figure}

By \cite[Proposition~1.3]{FKY22} again, estimate \eqref{eq:mainlowerest} is also sharp up to the logarithmic factor.
We should remark that, in the case of collinear configurations
\[ P=\{t_1<t_2<\cdots<t_n\}\subseteq\R\equiv \R\times\{0\}^{d-1}, \]
an $\ell^p$-isometric copy of $rP$ in $\R^d$ is simply any $n$-tuple of points
\[ y_1=x+t_1 v,\ \ y_2=x+t_2 v,\ \ \ldots,\ \ y_n=x+t_n v \]
such that $\|v\|_p=r$; see Figure~\ref{fig:dilcopy}.
This is proved in Lemma~\ref{lem:collinear-copy} in the appendix.

Restricting our attention to integer values of $p$ might seem a bit artificial. It is primarily dictated by our proof via Proposition~\ref{proposition:main} below, which is concerned with degree-$p$ polynomials. Moreover, there is a plausible conjecture on this topic concerning \emph{arithmetic progressions}, which are similar copies of
\begin{equation}\label{eq:pis0n}
P = \{ 0, 1, 2, \ldots, n-1 \}.
\end{equation}

\begin{problem}[{\cite[Problem~4]{DK22}}]
Prove or disprove that for every integer $n\geq 3$ and every exponent $p\in[1,\infty)\setminus\{1,2,\ldots,n-1\}$ there exists a number $D(n,p)\in(0,\infty)$ such that in every dimension $d\geq D(n,p)$ the following holds: if $E\subseteq\mathbb{R}^d$ is a measurable set satisfying $\overline{\rho}(E)>0$, then there exists $r_0=r_0(n,p,d,E)\in(0,\infty)$ such that for every $r\geq r_0$ one can find $x,v\in\mathbb{R}^d$ satisfying 
\[ x, x+v,\ldots, x+(n-1)v \in E \quad\text{and}\quad \|v\|_p=r. \]
\end{problem}

So far, this problem has been solved positively only for $n=3$; see \cite{CMP17} and the generalizations in \cite{DKR18,DK21}. Even if it asks about a density theorem with parameters that are allowed to depend arbitrarily on $n$, it is already an indication that the lower estimate \eqref{eq:mainlowerest} is less likely to hold for non-integral $p$.

If we no longer require the obstructing pattern $P$ to be collinear, then $\ell^p$ density theorems become much easier for every real exponent $p\neq 2$.

\begin{theorem}
\label{theorem:fulllp}
Fix an integer $d\ge 1$ and a real number $p\in(1,\infty)$ such that $p\neq 2$.
Then for every integer $n\ge 2d+1$ there exist
\begin{itemize}
    \item a measurable set $E\subseteq\R^d$,
    \item an $n$-point configuration $P\subseteq\R^d$,
    \item and a sequence of positive numbers $(r_j)_{j}$ with $r_j\to\infty$,
\end{itemize}
such that $E$ contains no $\ell^p$-isometric copy of $r_jP$ for any $j$, while
\[ \textup{d}(E)=1-\frac{1}{n-2d+2}. \]
\end{theorem}

Here it becomes important to clarify what a general $\ell^p$-isometric copy is.
If $P\subseteq\R^d$ is finite and $r\in(0,\infty)$, we say that a set $Y\subseteq\R^d$ is an \emph{$\ell^p$-isometric copy of $rP$} if there exists a bijection $\phi\colon P\to Y$ such that
\[ \norm{\phi(u)-\phi(v)}_p=r\norm{u-v}_p \]
for all $u,v\in P$.

Note that \cite[Proposition~1.3]{FKY22} still applies, since it only uses translates of a given configuration, and translations are $\ell^p$ isometries for every $p$.
The following is a clear consequence of this fact and Theorem~\ref{theorem:fulllp}. It fully resolves the $p\neq2$ analogue of the main question.

\begin{corollary}
The smallest number $0\leq\rho_{\min}(d,n,p)\leq1$ that forces every measurable set $E \subseteq \R^d$ of upper Banach density $\overline{\rho}(E)>\rho_{\min}(d,n,p)$ to contain $\ell^p$-isometric copies of all sufficiently large dilates of every $n$-point configuration is of the form
\[ \rho_{\min}(d,n,p) = 1 - \frac1n + O_d\Big(\frac1{n^2}\Big) \]
for every $p\in(1,\infty)\setminus\{2\}$.
\end{corollary}

The configuration $P$ mentioned in Theorem~\ref{theorem:lp} is produced by a probabilistic thinning argument in Section~\ref{sec:number-theory}.
It is collinear, but our proof does not force it to be a consecutive arithmetic progression.
By contrast, Theorem~\ref{theorem:fulllp}, proved in Section~\ref{sec:fulllp}, becomes much simpler because it allows a non-collinear pattern $P$.

We will begin the proof of Theorem~\ref{theorem:lp} by reducing it to the following quantitative number-theoretic statement about a certain uniform density property for degree-$p$ polynomials modulo $1$. 

\begin{proposition}
\label{proposition:main}
Fix an integer $p\ge 1$. There exists a constant $K_p\in(0,\infty)$ with the following property. For every sufficiently large integer $n$ there exist a set $P\subseteq\Z$ with $n=\card P$ and a real number $A$ such that for every $B_1,\dots,B_{p-1}\in\R$ the set
\begin{equation}\label{eq:thepolynomialset}
\big\{(A k^p + B_{p-1}k^{p-1}+\cdots+B_1 k) \bmod 1 \,:\, k\in P\big\}\subseteq\T
\end{equation}
intersects every interval on $\T$ of length
\begin{equation}\label{eq:whatisepsilon}
\varepsilon_n = K_p \frac{\log n}{n}.
\end{equation}
\end{proposition}

In Section~\ref{sec:number-theory}, we will prove Proposition~\ref{proposition:main} by combining Weyl's inequality, the Erd\H{o}s--Tur\'an discrepancy bound, and a random pattern thinning argument.
The first two ingredients essentially mirror the approach used in \cite[Section~4]{FKY22} for $p=2$, while the third seems somewhat novel in this context. It was suggested to us by \emph{ChatGPT} 5.4 Pro \cite{chatgpt}.
Afterwards, Yann Bugeaud informed us of several papers in Diophantine approximation theory \cite{PS10,Mos09,Mos10a,Mos10b,BM11,BM12,GM23}, where the probabilistic method is used only at the cost of losing a logarithmic factor in an approximation estimate.

\subsection*{Notation}
We adopt standard asymptotic notation:
\begin{itemize}
\item the \emph{Bachmann--Landau convention}, writing $f(x) = O(g(x))$ when there exists an absolute constant $C > 0$ such that $|f(x)| \leq C g(x)$ for all $x$ in the common domain of $f$ and $g$, and
\item \emph{Vinogradov's notation} $f(x) \ll g(x)$ as a synonym for $f(x) = O(g(x))$. 
\end{itemize}
When the implied constant depends on a set of parameters $\mathcal{P}$, we emphasize this notationally with subscripts, i.e., by writing $f(x) = O_{\mathcal{P}}(g(x))$ and $f(x) \ll_{\mathcal{P}} g(x)$.
We also write $f(x) \asymp g(x)$ when both $f(x) \ll g(x)$ and $g(x) \ll f(x)$ hold.

Let $\T=\R/\Z$ denote the one-dimensional torus. It inherits the addition from $\R$, while the Lebesgue measure is transferred from $[0,1)$ via the bijection $[0,1)\to\T$, $t\mapsto t+\Z$. The canonical projection $\R\to\T$ is often rather written as $t\mapsto t\bmod 1 := t+\Z$. 
Let us also write $\|x\| := \dist(t,\Z)$ for every ``point'' $x=t+\Z$ on $\T$, thus defining the function $\|\cdot\|\colon\T\to[0,1/2]$. Clearly, $\|x+y\|\leq\|x\|+\|y\|$ for any $x,y\in\T$. In particular, $(x,y)\mapsto\|x-y\|$ is a metric on $\T$, the so-called ``arclength'' distance on the torus. An interval on $\T$ is an image of an interval $I\subseteq \R$ via the canonical projection.

The logarithm is always understood to have base $e$. We respectively write $\lfloor t\rfloor$ and $\lceil t\rceil$ for the floor and the ceiling of a real number $t$. We will keep writing just $|E|$ for the Lebesgue measure of a set $E\subseteq\R^d$. The cardinality of a finite set $P\subseteq\R^d$ will instead be written $\card P$. The standard basis of $\R^d$ is denoted $\mathbbm{e}_1,\ldots,\mathbbm{e}_d$.


\section{Reduction of Theorem~\ref{theorem:lp} to Proposition~\ref{proposition:main}}
\label{sec:reduction}

Theorem~\ref{theorem:l2} becomes a special case of Theorem~\ref{theorem:lp} after taking into account its last claim about the uniformity in $d$ of the constants $C_{d,p}$ for a fixed even $p$ (and in particular for $p=2$). Thus, we only need to show the second theorem and we will do so assuming Proposition~\ref{proposition:main}.
We will split the proof of Theorem~\ref{theorem:lp} into two parts according to the parity of the number $p$. The case of even $p$ is more direct and similar to the approach in \cite{FKY22}. 
The case of odd $p$ requires an additional idea of inserting $\pm1$ signs to overcome the fact that $x\mapsto \|x\|_p^p$ is no longer a polynomial in $x=(x_1,\ldots,x_d)$.

Throughout this section we assume that $K_p\in(0,\infty)$, $\varepsilon_n\in(0,1)$, $P\subseteq\Z$, and $A\in\R$ are given by Proposition~\ref{proposition:main} for a fixed sufficiently large $n$, and we also set 
\[ r_j := (A+j)^{1/p} \]
for all positive integers $j$ sufficiently large that $A+j>0$.

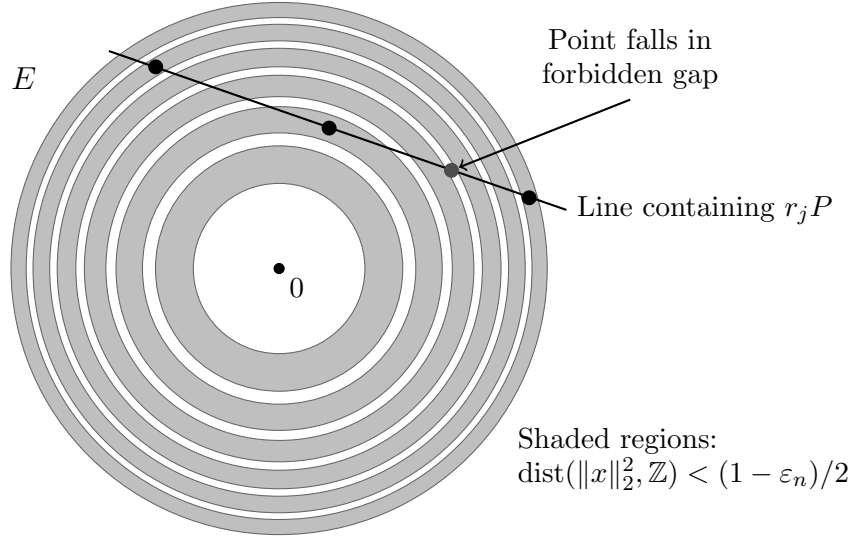
\begin{figure}
\begin{flushleft}
\hspace*{0.24\textwidth}
\begin{tikzpicture}[scale=1.4]
    \foreach \k in {1, 2, 3, 4, 5, 6} {
        \pgfmathsetmacro{\rIn}{sqrt(\k - 0.35)}
        \pgfmathsetmacro{\rOut}{sqrt(\k + 0.35)}
        \fill[black!25, even odd rule] (0,0) circle (\rOut) (0,0) circle (\rIn);
        \draw[black!60, thin] (0,0) circle (\rOut);
        \draw[black!60, thin] (0,0) circle (\rIn);
    }
    \fill[black] (0,0) circle (1.5pt) node[below right, black] {$0$};
    \draw[thick, black] (-1.6, 2.06) -- (2.7, 0.555) node[right, black] {Line containing $r_j P$};
    \fill[black] (-1.16, 1.91) circle (2pt);
    \fill[black] (0.47, 1.33) circle (2pt);
    \fill[black!70] (1.62, 0.93) circle (2pt); 
    \draw[->, thick, black, shorten >= 3pt] (3.3, 1.6) node[above, text width=3cm, align=center] {Point falls in \\ forbidden gap} -- (1.62, 0.93);
    \fill[black] (2.35, 0.67) circle (2pt);
    \node[black, font=\large] at (-2.4, 1.8) {$E$};
    \node[text width=4.5cm, align=left] at (3.85, -1.8) 
        {Shaded regions: \\ $\operatorname{dist}(\|x\|_2^2,\mathbb{Z}) < (1-\varepsilon_n)/2$};
\end{tikzpicture}
\end{flushleft}
\caption{The set $E$ for $p=2$.}
\label{fig:annuli}
\end{figure}

\begin{proof}[Proof of Theorem~\ref{theorem:lp} for even $p$]
As in \cite[Section~4]{FKY22}, we generalize Bourgain's construction with annuli around the origin \cite[Section~1]{Bou86} and define the set
\[ E := \Big\{x\in\R^d : \dist\bigl(\|x\|_p^p,\Z\bigr) < \frac{1-\varepsilon_n}{2}\Big\}. \]
Its instance for $p=2$ is illustrated in Figure~\ref{fig:annuli}.
For each fixed $(x_2,\dots,x_d)\in[-R/2,R/2]^{d-1}$, formula \eqref{eq:onedens2} from Lemma~\ref{lem:one-variable-density} applied with $\sigma=1$ to the interval
\[ I=\Big(-\frac{1-\varepsilon_n}{2},\frac{1-\varepsilon_n}{2}\Big)-\|(x_2,\dots,x_d)\|_p^p \]
of length $1-\varepsilon_n$ gives
\[ \Big|\Big\{x_1\in\Big[-\frac{R}{2},\frac{R}{2}\Big] : (x_1,x_2,\dots,x_d)\in E\Big\}\Big|
= (1-\varepsilon_n)R + O_p(1). \]
Integrating in $x_2,\dots,x_d$ we obtain
\[ \Big|E\cap\Big[-\frac{R}{2},\frac{R}{2}\Big]^d\Big| = (1-\varepsilon_n)R^d + O_{d,p}(R^{d-1}), \]
so
\[ \textup{d}(E)=1-\varepsilon_n = 1- K_p \frac{\log n}{n}. \]

Now suppose, for contradiction, that the set $E$ contains an $\ell^p$-isometric copy of $r_jP$ for some $j$.
By Lemma~\ref{lem:collinear-copy}, this copy has the form
\begin{equation}\label{eq:whatisY}
Y = \{x + r_j k v : k\in P\}
\end{equation}
for some 
\begin{equation}\label{eq:whatisXV}
x=(x_1,\ldots,x_d),\quad v=(v_1,\ldots,v_d) 
\end{equation}
with $\|v\|_p=1$.
For $k\in P$, use the binomial theorem to expand
\[ \|x+r_jkv\|_p^p 
= \sum_{i=1}^{d} \sum_{l=0}^{p} \binom{p}{l} x_i^{p-l} r_j^l k^l v_i^l
= r_j^p \|v\|_p^p k^p + \sum_{l=0}^{p-1} B_l k^l
= (A+j) k^p + \sum_{l=0}^{p-1} B_l k^l, \]
where
\[ B_l := \binom{p}{l} r_j^l \sum_{i=1}^{d} x_i^{p-l} v_i^l \quad \text{for } 0\leq l\leq p-1. \]
Since $Y\subseteq E$, all points 
\[ \|x+r_jkv\|_p^p \bmod 1, \quad k\in P,\]
lie inside a single interval of length $1-\varepsilon_n$ on $\T$ by the definition of $E$.
After subtracting the constant term $B_0$ and the integer expression $jk^p$, we conclude the same for the set \eqref{eq:thepolynomialset}, which contradicts Proposition~\ref{proposition:main}.
Therefore, $E$ contains no $\ell^p$-isometric copy of $r_jP$ for any index $j$.
\end{proof}

\begin{proof}[Proof of Theorem~\ref{theorem:lp} for odd $p$]
For every $\sigma=(\sigma_1,\dots,\sigma_d)\in\{-1,1\}^d$ write
\[ F_\sigma(x_1,\dots,x_d) := \sigma_1 x_1^p + \cdots + \sigma_d x_d^p \]
and define
\[ E_\sigma := \Big\{x\in\R^d : \dist\bigl(F_\sigma(x),\Z\bigr) < \frac{1-\varepsilon_n}{2}\Big\}. \]
Then set
\[ E := \bigcap_{\sigma\in\{-1,1\}^d} E_\sigma. \]
Fix $\sigma\in\{-1,1\}^d$ and $(x_2,\dots,x_d)\in[-R/2,R/2]^{d-1}$. Lemma~\ref{lem:one-variable-density} gives
\[ \Big|\Big\{x_1\in\Big[-\frac{R}{2},\frac{R}{2}\Big] : (x_1,x_2,\dots,x_d)\in E_\sigma\Big\}\Big| = (1-\varepsilon_n)R + O_p(1). \]
Integrating in $x_2,\dots,x_d$ and considering complements we get
\[ \Big|\Big[-\frac{R}{2},\frac{R}{2}\Big]^d\setminus E\Big|
\le \sum_{\sigma\in\{-1,1\}^d} \Big|\Big[-\frac{R}{2},\frac{R}{2}\Big]^d\setminus E_\sigma\Big|
\le 2^d \varepsilon_n R^d + O_{d,p}(R^{d-1}), \]
which quickly yields
\[ \underline{\textup{d}}(E) \geq 1 - 2^d \varepsilon_n = 1 - 2^d K_p \frac{\log n}{n}. \]

Suppose that, for some $j$, the set $E$ contains an $\ell^p$-isometric copy of $r_jP$, which, by Lemma~\ref{lem:collinear-copy}, is necessarily of the form \eqref{eq:whatisY} for some $d$-tuples \eqref{eq:whatisXV} satisfying $\|v\|_p=1$.
Choose $\sigma=(\sigma_1,\dots,\sigma_d)$ by
\[ \sigma_i := \begin{cases}
1 & \text{if } v_i\geq0,\\
-1 & \text{if } v_i<0
\end{cases} \]
for each $i=1,2,\ldots,d$.
Then
\[ F_\sigma(x+r_jkv)
= \sum_{i=1}^{d} \sigma_i \sum_{l=0}^{p} \binom{p}{l} x_i^{p-l} r_j^l k^l v_i^l
= A k^p + \sum_{l=0}^{p-1} B_l k^l + j k^p \]
for suitable real numbers $B_0,B_1,\ldots,B_{p-1}$, where the leading coefficient (the coefficient of $k^p$) is computed as
\[ r_j^p \sum_{i=1}^d \sigma_i v_i^p = r_j^p \sum_{i=1}^d |v_i|^p
= r_j^p \|v\|_p^p = A+j. \]
Since $Y\subseteq E\subseteq E_\sigma$, all values $F_\sigma(x+r_jkv)\bmod 1$ lie inside one interval of length $1-\varepsilon_n$ on $\T$.
Translating by $B_0$, we conclude that the set \eqref{eq:thepolynomialset} also lies in an interval on $\T$ of length $1-\varepsilon_n$, contradicting Proposition~\ref{proposition:main} again.
\end{proof}


\section{Proof of Proposition~\ref{proposition:main}}
\label{sec:number-theory}

\emph{Discrepancy} of $N$ points $(x_k)_{k=0}^{N-1}$ on $\T$ is defined as the number
\[ D_N\big((x_k)_{k=0}^{N-1}\big) := \sup_{\substack{I\subseteq\T\\I\,\text{interval}}} \bigg| \frac{\mathop{\textup{card}}\big\{ k\in\{0,1,\ldots,N-1\} :  x_k \in I \big\}}{N} - |I| \bigg|. \]
Write
\[ \e(x):=\exp(2\pi i x). \]
A way of estimating the discrepancy is via the exponential sums $\sum_{k=0}^{N-1} \e(m x_k)$ and the well-known \emph{Erd\H{o}s--Tur\'{a}n inequality} \cite[Theorem~2.5]{KN74},
\begin{equation}\label{eq:ErdosTuran}
D_N\big((x_k)_{k=0}^{N-1}\big) \ll \frac{1}{M} + \sum_{m=1}^{M} \frac{1}{m} \Big| \frac{1}{N} \sum_{k=0}^{N-1} \e(m x_k) \Big|,
\end{equation}
where $M$ is any (appropriately chosen) positive integer and the implicit constant is an absolute one.
The proof will also use the classical \emph{Weyl's inequality} \cite[Lemma~2.4]{Vaughan97} in the following form. Let $\alpha$ be a real number for which there exist coprime integers $a,q$ such that $|\alpha-a/q|\le q^{-2}$, let 
\[ f(t)=\alpha t^p + \alpha_{p-1} t^{p-1} + \cdots + \alpha_1 t + \alpha_0 \]
be a real polynomial of degree $p\geq1$ with the leading coefficient $\alpha$, and let $\eta>0$ be arbitrary.
Then for every positive integer $N$ we have
\begin{equation}\label{eq:Weyls_ineq}
\Big|\sum_{k=0}^{N-1} \e\bigl(f(k)\bigr)\Big|
\ll_{p,\eta}
N^{1+\eta}\Bigl( \frac{1}{q} + \frac{1}{N} + \frac{q}{N^{p}} \Bigr)^{1/2^{p-1}}.
\end{equation}
Note that the bound is uniform in the lower-order coefficients $\alpha_0,\alpha_1,\ldots,\alpha_{p-1}$.
The main novelty of the proof (compared to \cite{FKY22}) is that discrepancy will be estimated for a significantly larger set of points, and then we will randomly choose its $n$-element subset, making sure that the desired property is satisfied with a positive probability. 

\begin{proof}[Proof of Proposition~\ref{proposition:main}]
Choose a prime $Q$ with
\[ n^{2^p} < Q < 2n^{2^p}, \]
which is possible by Bertrand's postulate for all sufficiently large $n$. Set
\[ A := \frac{1}{Q}. \]
For each fixed vector
\[ B=(B_1,\dots,B_{p-1})\in\R^{p-1} \]
and for every $k\in\Z$ 
\[ \Big( A k^p + B_{p-1}k^{p-1} + \cdots + B_1 k \Big) \bmod 1 \]
coincides with
\[ x_k(B) := \Big( \frac{k^p}{Q} + B_{p-1}k^{p-1}+\cdots+B_1k \Big) \bmod 1. \]

We begin by considering (a quite large) collection $\big(x_k(B)\big)_{k=0}^{Q-1}$ of $Q$ points on $\T$.
Let us estimate the discrepancy of the full collection.
For each integer $1\le m\le Q-1$ we apply Weyl's inequality \eqref{eq:Weyls_ineq} with $N=Q$, $\alpha=a/q=m/Q$ (which is already a reduced fraction), $\eta=2^{-p}$, and the polynomial
\[ f(t) = \frac{m}{Q}t^p + mB_{p-1}t^{p-1}+\cdots+mB_1t. \]
It gives us
\begin{equation}
\label{eq:weyl-Sm}
\Big| \sum_{k=0}^{Q-1} \e\big(m\,x_k(B)\big) \Big| \ll_p Q^{1-2^{-p}}
\end{equation}
uniformly in $m$ and in $B$.
The discrepancy of our collection is then controlled by the Erd\H{o}s--Tur\'{a}n inequality \eqref{eq:ErdosTuran} with $N=Q$ and $M=Q-1$ as
\[ D_Q\Big(\big(x_k(B)\big)_{k=0}^{Q-1}\Big)
\ll \frac{1}{Q-1} + \sum_{m=1}^{Q-1} \frac{1}{m}\Big|\frac{1}{Q} \sum_{k=0}^{Q-1} \e\big(m\,x_k(B)\big) \Big|, \]
which, in combination with \eqref{eq:weyl-Sm} and $Q\asymp n^{2^p}$, yields
\[ D_Q\Big(\big(x_k(B)\big)_{k=0}^{Q-1}\Big)
\ll_p \frac{1}{Q} + \frac{1}{Q^{2^{-p}}}\sum_{m=1}^{Q-1}\frac{1}{m}
\ll_p \frac{\log Q}{Q^{2^{-p}}} \ll_p \frac{\log n}{n} \]
uniformly in $B$.
In other words, if $K_p$ is chosen large enough and $\varepsilon_n$ is as in \eqref{eq:whatisepsilon}, then
\[ D_Q\Big(\big(x_k(B)\big)_{k=0}^{Q-1}\Big) \le \frac{\varepsilon_n}{10} \qquad\text{for every }B\in\R^{p-1}. \]
Hence, every interval $I\subseteq\T$ of length $9\varepsilon_n/10$ contains at least
\[ \Big(\frac{9\varepsilon_n}{10}-\frac{\varepsilon_n}{10}\Big)Q
= \frac{4}{5}\varepsilon_n Q \]
indices $k\in\{0,\dots,Q-1\}$ such that $x_k(B)\in I$.

Observe that only the classes of the coefficients $B_1,\ldots,B_{p-1}$ modulo $1$ matter. For $i=1,\dots,p-1$ choose a torus subdivision $\mathcal{B}_i\subseteq\T$ (which is just a finite ordered collection of points) of mesh
\[ \Delta_i := \frac{\varepsilon_n}{100p\,Q^i} \]
and cardinality $|\mathcal{B}_i| \ll Q^i/\varepsilon_n$. Let
\[ \mathcal{B}:=\mathcal{B}_1\times\cdots\times\mathcal{B}_{p-1}\subseteq \T^{p-1}, \]
so that the size of this $(p-1)$-dimensional grid satisfies
\begin{equation}\label{eq:B-net-size}
|\mathcal{B}| \ll_p \frac{Q^{1+2+\cdots+(p-1)}}{\varepsilon_n^{p-1}}
= \frac{Q^{p(p-1)/2}}{\varepsilon_n^{p-1}}.
\end{equation}
Let $\mathcal{I}$ be the family of torus intervals of length $9\varepsilon_n/10$ whose left endpoints are the integer multiples of $\varepsilon_n/100$ that belong to $[0,1)$. Namely, $\mathcal{I}$ consists of the intervals
\[ \Big[\frac{\varepsilon_n l}{100},\frac{\varepsilon_n l}{100}+\frac{9\varepsilon_n}{10}\Big) \bmod 1, \quad l=0,1,2,\ldots,\Big\lceil\frac{100}{\varepsilon_n}\Big\rceil-1. \]
Clearly,
\begin{equation}
\label{eq:I-net-size}
|\mathcal{I}| \ll \frac{1}{\varepsilon_n}.
\end{equation}
Fix $b\in\mathcal{B}$ and $I\in\mathcal{I}$, and define
\[ R_{b,I} := \big\{ k\in\{0,\dots,Q-1\} \,:\, x_k(b)\in I \big\}. \]
By the previous part of the proof,
\begin{equation}\label{eq:firstb} 
|R_{b,I}|\ge \frac{4}{5}\varepsilon_n Q. 
\end{equation}

The next step is thinning of the set $\{0,1,2,\ldots,Q-1\}$ to exactly $n$ indices. Choose $P$ uniformly at random among all of its $n$-element subsets. Thanks to \eqref{eq:firstb}, the probability that a random $n$-subset $P$ is disjoint from $R_{b,I}$ is
\[ \frac{\binom{Q-|R_{b,I}|}{n}}{\binom{Q}{n}}
\leq \Big(1-\frac{|R_{b,I}|}{Q}\Big)^n 
\leq \Big(1-\frac{4}{5}\varepsilon_n\Big)^n
\leq e^{-4\varepsilon_n n/5}. \]
By the union bound, the probability that a random $n$-subset $P$ is disjoint from $R_{b,I}$ for at least one pair $(b,I)\in\mathcal{B}\times\mathcal{I}$ is now at most
\[ |\mathcal{B}|\,|\mathcal{I}| \,e^{-4\varepsilon_n n/5}. \]
From \eqref{eq:B-net-size}, \eqref{eq:I-net-size}, \eqref{eq:whatisepsilon}, and $Q\asymp n^{2^p}$, this is
\[ \ll_p \frac{Q^{p(p-1)/2}}{\varepsilon_n^p} \,e^{-4\varepsilon_n n/5}
\asymp_p \exp \bigg( \Big(p(p-1)2^{p-1} + p - \frac{4}{5}K_p\Big) \log n - p \log\log n \bigg). \]
As soon as $K_p$ is sufficiently large, the last number will be strictly less than $1$ for every large enough $n$.
In other words, there exists at least one $P\subseteq\{0,1,2,\ldots,Q-1\}$ with $n=\card P$ such that 
\[ P\cap R_{b,I}\neq\emptyset, \]
i.e.,
\begin{equation}\label{eq:firstx} 
\{x_k(b) : k\in P\}\cap I\neq\emptyset 
\end{equation}
for every $b\in\mathcal{B}$ and every $I\in\mathcal{I}$.

Finally, for the obtained $P\subseteq\Z$, we take care of arbitrary $(p-1)$-tuples of coefficients $B$ by approximating them with points from the above discrete grid $\mathcal{B}$. Namely, take an arbitrary $B=(B_1,\dots,B_{p-1})\in\T^{p-1}$ and an arbitrary interval $J\subseteq\T$ of length $\varepsilon_n$. 
Choose $b=(b_1,\dots,b_{p-1})\in\mathcal{B}$ so that
\[ \|B_i-b_i\|\le \Delta_i \quad\text{for } i=1,2,\ldots,p-1. \]
Then, for every $k\in P$,
\[ \|x_k(B)-x_k(b)\| \le \sum_{i=1}^{p-1} \|(B_i-b_i)k^i\| 
\le \sum_{i=1}^{p-1} \Delta_i Q^i \le \frac{\varepsilon_n}{100}. \]
Take an interval $J'\subseteq J\subseteq\T$ with the same center as $J$ and length $96\varepsilon_n/100$. It contains some $I\in\mathcal{I}$. By \eqref{eq:firstx} there exists $k\in P$ such that $x_k(b)\in I\subseteq J'$, and then we have $x_k(B)\in J$.
Therefore,
\[ \{x_k(B) : k\in P\}\cap J\neq\emptyset, \]
so, for every $B_1,\dots,B_{p-1}\in\R$, the set \eqref{eq:thepolynomialset} intersects every interval of length $\varepsilon_n$ on $\T$. This proves the proposition.
\end{proof}

Stronger Vinogradov's mean value theorem or decoupling forms of Weyl's inequality would improve the required power of $n$ in $Q$, but they would not remove the final logarithm, which comes from the random thinning step.


\section{Proof of Theorem~\ref{theorem:fulllp}}
\label{sec:fulllp}

The proof presented in this section is geometric and does not use number-theoretic arguments from the previous section.

\begin{proof}[Proof of Theorem~\ref{theorem:fulllp}]
Define
\[ P:= \big\{ k\mathbbm{e}_1 \,:\, k\in\{-1,0,1,2,\ldots,n-2d \} \big\}
\cup \{\mathbbm{e}_2,\ldots,\mathbbm{e}_d\} 
\cup \{-\mathbbm{e}_2,\ldots,-\mathbbm{e}_d\} \subseteq \R^d, \]
so that $\card P = n$; see Figure~\ref{fig:lpset} for the special case $d=2$ and $n=8$.
Also set
\[ \varepsilon := \frac{1}{n-2d+2} \]
and
\[ E := \big\{ x =(x_1,\ldots,x_d)\in \R^d \,:\, (x_1+\cdots+x_d) \bmod 1 \in [0,1-\varepsilon) \bmod 1 \big\}. \]
For each fixed $(x_2,\dots,x_d)\in[-R/2,R/2]^{d-1}$, the set of all $x_1\in[-R/2,R/2]$ for which $x_1 \bmod 1$ falls into the prescribed interval on $\T$ of length $1-\varepsilon$ has measure $(1-\varepsilon)R + O(1)$. Integration in $x_2,\dots,x_d$ gives
\[ \Big|E\cap\Big[-\frac{R}{2},\frac{R}{2}\Big]^d\Big| = (1-\varepsilon)R^d + O_d(R^{d-1}), \]
so
\[ \textup{d}(E) = 1-\varepsilon = 1-\frac{1}{n-2d+2}. \]

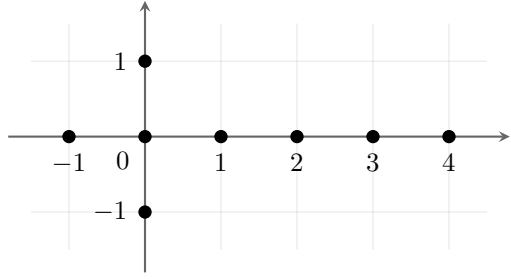
\begin{figure}
\begin{center}
\begin{tikzpicture}[scale=1, >=stealth]
    \draw[black!10, very thin, step=1.0] (-1.5,-1.5) grid (4.5,1.5);
    \draw[->, thick, black!60] (-1.8,0) -- (4.8,0);
    \draw[->, thick, black!60] (0,-1.8) -- (0,1.8);
    \foreach \x in {-1, 1, 2, 3, 4} {
        \draw[thick, black!60] (\x, -0.05) -- (\x, 0.05) node[below=4pt, text=black, font=\small] {$\x$};
    }
    \foreach \y in {-1, 1} {
        \draw[thick, black!60] (-0.05, \y) -- (0.05, \y) node[left=4pt, text=black, font=\small] {$\y$};
    }
    \node[below left=2pt, text=black, font=\small] at (0,0) {$0$};
    \foreach \x in {-1, 0, 1, 2, 3, 4} {
        \fill[black] (\x, 0) circle (2.5pt);
    }
    \fill[black] (0, 1) circle (2.5pt);
    \fill[black] (0, -1) circle (2.5pt);
\end{tikzpicture}
\end{center}
\caption{Configuration $P$ in the proof of Theorem~\ref{theorem:fulllp}.}
\label{fig:lpset}
\end{figure}

Define
\[ r_j:=j+\varepsilon \]
for every $j\geq1$. It remains to prove that $E$ contains no $\ell^p$-isometric copy of $r_j P$.
Assume, to the contrary, that for some $j$ a set $Y\subseteq E$ is an $\ell^p$-isometric copy of $r_j P$.
By Lemma~\ref{lem:collinear-copy}, there exist $x,u\in\R^d$ with $\|u\|_p=1$ such that points
\[ r_j k\mathbbm{e}_1, \quad k=-1,0,1,2,\ldots,n-2d, \]
are mapped to
\begin{equation}\label{eq:pointsy}
y_k := x + r_j k u, \quad k=-1,0,1,2,\ldots,n-2d 
\end{equation}
via the aforementioned $\ell^p$ isometry.
We claim that 
\begin{equation}\label{eq:whatisu}
\text{either } u=\mathbbm{e}_l \text{ or } u=-\mathbbm{e}_l \text{ for some } l\in\{1,2,\ldots,d\}. 
\end{equation}
The claim is trivial for $d=1$, so assume that $d\ge 2$.
Let the images of the remaining points of $r_j P$,
\[ r_j \mathbbm{e}_2,\ldots, r_j \mathbbm{e}_d, -r_j \mathbbm{e}_2,\ldots, -r_j \mathbbm{e}_d, \]
via the same $\ell^p$ isometry be denoted, respectively, by
\[ z_2^+, \ldots, z_d^+, z_2^-, \ldots, z_d^-. \]
Also write
\[ v_i^{\pm} := \frac{1}{r_j} (z_i^{\pm}-x) \]
for $i=2,\ldots,d$, so that $\|v_i^+\|_p=\|v_i^-\|_p=1$ for each index $i$.

First,
\[ \norm{v_i^+-v_i^-}_p = \frac{1}{r_j} \norm{z_i^+-z_i^-}_p = \norm{ \mathbbm{e}_i - (-\mathbbm{e}_i) }_p = 2. \]
Hence, 
\[ 2=\norm{v_i^+-v_i^-}_p\le \norm{v_i^+}_p+\norm{-v_i^-}_p=2. \]
Equality therefore holds in the triangle inequality for $v_i^+$ and $-v_i^-$.
Since $\ell^p$ is strictly convex for $1<p<\infty$ \cite[Chapter~5]{Meg98}, this forces 
\begin{equation}\label{eq:vpm}
v_i^-=-v_i^+
\end{equation} 
for every index $i\geq2$.

Next, the $\ell^p$ isometry turns equalities
\[ \norm{\mathbbm{e}_i-\mathbbm{e}_1}_p^p = \norm{\mathbbm{e}_i-(-\mathbbm{e}_1)}_p^p=2 \]
from the model configuration $P$ into
\[ \norm{z_i^+-y_1}_p^p = \norm{z_i^+-y_{-1}}_p^p = 2r_j^p, \]
i.e.,
\[ \norm{z_i^+-(x+r_ju)}_p^p = \norm{z_i^+-(x-r_ju)}_p^p = 2r_j^p, \]
so that we have
\[ \norm{v_i^{+}-u}_p^p = \norm{v_i^{+}+u}_p^p=2. \]
Therefore,
\[ \norm{v_i^{+}-u}_p^p + \norm{v_i^{+}+u}_p^p = 4 = 2 \big(\norm{v_i^{+}}_p^p + \norm{u}_p^p\big), \]
so Lemma~\ref{lem:scalar} implies that $v_i^+$ and $u$ share no indices with nonzero coordinates for every index $i\geq2$.

Moreover, take indices $2\leq i<m\leq d$.
Since in the model configuration we have
\[ \norm{\mathbbm{e}_i-\mathbbm{e}_m}_p^p =\norm{\mathbbm{e}_i-(-\mathbbm{e}_m)}_p^p=2, \]
the $\ell^p$ isometry turns this into
\[ \norm{v_i^+-v_m^+}_p^p=\norm{v_i^+ -v_m^-}_p^p=2, \]
i.e., thanks to \eqref{eq:vpm},
\[ \norm{v_i^+-v_m^+}_p^p=\norm{v_i^+ +v_m^+}_p^p=2. \]
Since
\[ \norm{v_i^{+}-v_m^+}_p^p + \norm{v_i^{+}+v_m^+}_p^p = 4 = 2 \big(\norm{v_i^{+}}_p^p + \norm{v_m^+}_p^p\big), \]
the equality part of Lemma~\ref{lem:scalar} applies again and shows that $v_i^+$ and $v_m^+$ have no indices with nonzero coordinates in common.

We have therefore produced $d$ nonzero $\ell^p$-unit vectors, 
\[ u,v_2^+,\dots,v_d^+, \]
with pairwise disjoint index-sets of nonzero coordinates.
Since there are only $d$ coordinates available, each index-set must be a singleton and these singletons must be distinct.
In particular, $u=\mathbbm{e}_l$ or $u=-\mathbbm{e}_l$ for some $l\in\{1,\dots,d\}$, which confirms the claim \eqref{eq:whatisu}.

Now we return to the points \eqref{eq:pointsy}, which are all supposed to lie in $E$, i.e., there exist an index $l$ and a sign $\sigma\in\{-1,1\}$ such that
\[ x + \sigma r_j k \mathbbm{e}_l \in E \quad\text{for } k=-1,0,1,\ldots,n-2d, \]
i.e.,
\[ (x_1 + \cdots + x_d + \sigma r_j k) \bmod 1 \in [0,1-\varepsilon) \bmod 1 \quad\text{for } k=-1,0,1,\ldots,n-2d. \]
Consequently,
\[ \Big\{ \frac{k}{n-2d+2} \bmod 1 \,:\, k\in\{-1,0,1,2,\ldots,n-2d\} \Big\} \]
is contained in a half-open interval on $\T$ of length $1-1/(n-2d+2)$, which is impossible, because this is a set of $n-2d+2$ equally spaced points on $\T$.
\end{proof}


\section{An elementary approach}
\label{sec:elementary}

It is sufficient to establish Proposition~\ref{proposition:main} in the case $p=2$ and with $\varepsilon_n$ replaced by $10n^{-1/2}$. Afterwards, one can use the same construction of the set $E$ from the first proof in Section~\ref{sec:reduction}.
We simply choose the arithmetic progression \eqref{eq:pis0n} as the pattern $P$. We also assume that $n\geq4$. Take $m=\lfloor n^{1/2}\rfloor$ and split $\{0,1,2,\ldots,m^2-1\}$ into $m$ integer blocks of length $m$:
\[ \big\{im+l \,:\, i,l\in\{0,1,\ldots,m-1\} \big\} \subseteq P. \]
Also set
\[ A := \frac{1}{m^2} \]
and fix an arbitrary $B\in\R$. 

We first select an index $i\in\{0,1,\ldots,m-1\}$ such that 
\[ \Big( B + \frac{2i}{m} \Big) \bmod 1 \in \Big[\frac{1}{m},\frac{3}{m}\Big) \bmod 1 \]
and let $\theta\in[1/m,3/m)$ coincide with $B+2i/m$ modulo $1$.
For this fixed $i$, we have
\begin{align*}
& \big( A (im+l)^2 + B (im+l) \big) \bmod 1 \\
& = \Big( B i m + \Big( B + \frac{2i}{m} \Big) l + \frac{l^2}{m^2} \Big) \bmod 1 \\ 
& = t_l \bmod 1, 
\end{align*}
where $t_l\in\R$ is defined as
\[ t_l := B i m + \theta l + \frac{l^2}{m^2} \]
for every $l\in\{0,1,\ldots,m-1\}$.
Then, for $0\leq l\leq m-2$,
\[ t_{l+1} - t_{l} = \theta + \frac{2l+1}{m^2} \in \Big[\frac{1}{m},\frac{5}{m}\Big), \]
so consecutive terms among $t_0,t_1,\ldots,t_{m-1}$ are separated by less than $5/m$. Also,
\[ t_{m-1} - t_0 = \theta (m-1) + \frac{(m-1)^2}{m^2} \geq 1 - \frac{1}{m} + \Big(1-\frac{1}{m}\Big)^2 \geq 1, \]
so
\[ \big\{ t_l \bmod 1 \,:\, l\in\{0,1,\ldots,m-1\} \big\} \]
intersects every interval on $\T$ of length $5/m\leq 10n^{-1/2}$.
Consequently, its superset 
\[ \big\{ (A k^2 + B k) \bmod 1 \,:\, k\in P \big\} \]
appearing in \eqref{eq:thepolynomialset} has the same property and we are done.


\appendix

\section{Auxiliary lemmas}

We will need several simple auxiliary results. Each of them is either entirely elementary, or very standard, so they did not find a place in the main text.

\begin{lemma}
\label{lem:one-variable-density}
Fix an integer $p\geq2$. For every interval $I\subseteq\T$, every $\sigma\in\{-1,1\}$, and every $R\geq 1$ we have
\begin{equation}\label{eq:onedens2}
\Big|\Big\{t\in\Big[-\frac{R}{2},\frac{R}{2}\Big] \,:\, \sigma t^p \bmod 1 \in I \Big\}\Big| = |I|R + O_p(1) 
\end{equation}
uniformly in $I$.
\end{lemma}

\begin{proof}
By splitting $[-R/2,R/2]$ into positive and negative numbers and doubling $R$, one reduces \eqref{eq:onedens2} to 
\begin{equation}\label{eq:onedens1}
\big|\big\{t\in[0,R] \,:\, \sigma t^p \bmod 1 \in I \big\}\big| = |I|R + O_p(1).
\end{equation}
By reflecting the interval $I$, if needed, we can assume that $\sigma=1$. Finally, it is enough to consider $I=[0,a]$ for some $a\in(0,1)$.
The set appearing on the left-hand side of \eqref{eq:onedens1} is now
\begin{equation}\label{eq:theauxsetest}
\Big(\bigcup_{m=0}^{\infty} \big[m^{1/p},(m+a)^{1/p}\big]\Big) \cap [0,R] 
\end{equation}
and its measure equals
\[ \sum_{m=1}^{\lfloor R^p\rfloor} \big( (m+a)^{1/p}-m^{1/p} \big) + O_p(1). \]
Applying Cauchy's mean value theorem to the functions
\[ \varphi(x) := (1+ax)^{1/p} \quad\text{and}\quad \psi(x) := (1+x)^{1/p} \]
we obtain
\[ \frac{\varphi(1/m)-\varphi(0)}{\psi(1/m)-\psi(0)} = \frac{\varphi'(\xi)}{\psi'(\xi)} = a \Big(\frac{1+a\xi}{1+\xi}\Big)^{1/p-1} \]
for some $\xi\in(0,1/m)$, which implies that
\[ a \leq \frac{(m+a)^{1/p}-m^{1/p}}{(m+1)^{1/p}-m^{1/p}} \leq a \big(1 + O_p(m^{-1})\big) \]
and then also
\[ a \big( (m+1)^{1/p}-m^{1/p} \big) \leq (m+a)^{1/p}-m^{1/p} \leq a \big( (m+1)^{1/p}-m^{1/p} \big) + O(m^{-2+1/p}). \]
Summing this in $m=1,2,\ldots,\lfloor R^p\rfloor$ and telescoping we conclude that the measure of the set \eqref{eq:theauxsetest} is really $aR+O_p(1)$, as claimed.
\end{proof}

\begin{lemma}
\label{lem:collinear-copy}
Take $p\in(1,\infty)$ and $r\in(0,\infty)$, and let $P\subseteq\R$ be a finite set.
Suppose that an indexed collection of points $Y=\{y_t:t\in P\}\subseteq\R^d$ satisfies
\[ \|y_s-y_t\|_p = r |s-t| \]
for all $s,t\in P$.
Then there exist $x,v\in\R^d$ with $\|v\|_p=1$ such that
\[ y_t = x + r t v \]
for every $t\in P$.
\end{lemma}

In words, every $\ell^p$-isometric copy in $\R^d$ of a finite set $P\subseteq\R$ is again contained in a line.

\begin{proof}
Because of $p\in(1,\infty)$, the norm $\|\cdot\|_p$ is strictly convex \cite[Chapter~5]{Meg98}.
Let $a:=\min P$ and $b:=\max P$.
For any $t\in P$,
\[ \|y_b-y_a\|_p = r(b-a)
= r(b-t)+r(t-a) = \|y_b-y_t\|_p + \|y_t-y_a\|_p. \]
Thus, equality holds in the triangle inequality for the vectors $y_b-y_t$ and $y_t-y_a$.
Strict convexity implies that these two vectors are collinear and point in the same direction, so $y_t$ lies on the line segment joining $y_a$ and $y_b$.
Hence,
\[ y_t = y_a + \theta_t (y_b-y_a) \]
for some $\theta_t\in[0,1]$ and comparing distances we get
\[ \theta_t = \frac{\|y_t-y_a\|_p}{\|y_b-y_a\|_p} = \frac{t-a}{b-a}. \]
It remains to set
\[ v := \frac{y_b-y_a}{r(b-a)} \quad\text{and}\quad x := y_a - r a v. \qedhere \]
\end{proof}

\begin{lemma}[{Clarkson's inequalities \cite{Cla36}}]
\label{lem:scalar}
Take $1<p<\infty$, $p\neq 2$ and $x,y\in \R^d$. Then
\[ \|x+y\|_p^p+\|x-y\|_p^p
\begin{cases}
\ge 2\bigl(\|x\|_p^p+\|y\|_p^p\bigr) & \text{for } p>2,\\
\le 2\bigl(\|x\|_p^p+\|y\|_p^p\bigr) & \text{for } 1<p<2,
\end{cases} \]
with equalities if and only if $d$-tuples $x$ and $y$ share no indices with nonzero coordinates.
\end{lemma}

\begin{proof}
It is enough to show the following coordinate-wise statement:
\[ |a+b|^p+|a-b|^p
\begin{cases}
\ge 2\bigl(|a|^p+|b|^p\bigr) & \text{for } p>2,\\
\le 2\bigl(|a|^p+|b|^p\bigr) & \text{for } 1<p<2,
\end{cases} \]
with equalities if and only if $a=0$ or $b=0$.

Substitute
\[ s:=(a+b)^2 \quad\text{and}\quad t:=(a-b)^2. \]
If $p>2$, then $t\mapsto t^{p/2}$ is a convex function on $[0,\infty)$, so Jensen's inequality followed by the monotonicity of the $\ell^p$ norms gives
\[ s^{p/2}+t^{p/2} \geq 2\Big(\frac{s+t}{2}\Big)^{p/2}
=2(a^2+b^2)^{p/2} \geq 2(|a|^p+|b|^p). \]
If $p<2$, then $t\mapsto t^{p/2}$ is concave instead, so both inequalities above are reversed. 
For the equality claim, strict convexity/concavity forces $s=t$, i.e., $ab=0$.
\end{proof}


\section*{Declaration of AI usage}
\emph{ChatGPT} 5.4 Pro \cite{chatgpt} was used to suggest and draft approaches to Proposition~\ref{proposition:main}, proofread the text, and improve the presentation. \emph{Gemini} 3.1 Pro \cite{gemini} was used to draw figures and proofread the text. However, the main mathematical ideas, the proofs as presented, and the actual writing of the final manuscript are entirely the work of the authors. We independently checked all AI-produced claims, proofs, figures, and references, and assume full responsibility for the final content.


\section*{Acknowledgments and funding}
The authors are grateful to Yann Bugeaud for a useful discussion and for suggesting several relevant references.

V.\,K. was supported in part by the Croatian Science Foundation under the project HRZZ-IP-2022-10-5116 (\emph{FANAP}).
V.\,K. was also supported in part by the European Union -- NextGenerationEU through the National Recovery and Resilience Plan 2021--2026, via an institutional grant of the University of Zagreb Faculty of Science, IK IA 1.1.3, \emph{Impact4Math}.


\bibliography{density_large_configurations}{}
\bibliographystyle{amsplain}

\end{document}